\documentclass[a4paper]{amsart}

\usepackage{amsmath}
\usepackage{amssymb}
\usepackage{amsthm}

\newtheorem{theorem}{Theorem}

\newcommand{\Oh}{\mathrm{O}}
\newcommand{\oh}{\mathrm{o}}

\begin{document}

\title{The Shape of the Value Sets of Linear Recurrence Sequences}

\date{\today}

\author{Stefan Gerhold}

\thanks{This work was financially supported by CDG, BA-CA, AFFA, INRIA, and Microsoft Research.}

\address{Vienna University of Technology and
Microsoft Research-INRIA, Orsay
}

\email{sgerhold at fam.tuwien.ac.at}

\begin{abstract}
  We show that the closure of the value set of a real linear recurrence sequence
  is the union of a countable set and a finite collection of intervals.
  Conversely, any finite collection of closed intervals is the closure of the value set
  of some recurrence sequence.
\end{abstract}

\maketitle


\bigskip
\bigskip

A real linear recurrence sequence~$(a_n)_{n\geq0}$ of order~$h$ satisfies a relation of the form
\[
  a_{n+h} = \eta_1 a_{n+h-1} + \dots + \eta_{h-1} a_{n+1} + \eta_h a_n,\qquad n\geq 0,
\]
where the coefficients~$\eta_i$ and the initial values $a_0,\dots,a_{h-1}$ are real numbers.
It is well known that these are precisely the sequences that occur as Taylor coefficients
of rational functions with real coefficients, which explains part of their importance
in many areas of pure and applied mathematics.
In the vast literature on recurrence sequences (the monograph~\cite{EvPoSh03} is an excellent
entry point), the possible shape of the value set
\[
  \{ a_n : n \geq 0\}
\]
of such sequences seems not to have been worked out.
Typical recurrence sequences are, e.g., constant sequences, $2^n,(-1)^n,$ and $\cos n$.
The first three of these have discrete value sets, whereas the latter, which
satisfies the recurrence
\begin{equation}\label{eq:reccos}
  a_{n+2} = 2\cos(1) a_{n+1} - a_n, \qquad n\geq 0,
\end{equation}
is dense in the interval~$[-1,1]$, by a classical result in Diophantine approximation~\cite{Ca57}. This is
essentially the full picture, the precise statement being as follows.

\begin{theorem}
  The closure of the value set of a real linear recurrence sequence is the union of a countable set
  and a finite collection of intervals. Conversely, any finite non-empty collection of closed intervals
  is the closure of the value set of a recurrence sequence.
\end{theorem}
\begin{proof}
  It is well known that any real recurrence sequence has a representation
  \[
    a_n = |\alpha|^n(n^db_n + \Oh(n^{d-1})), \qquad n\to\infty,
  \]
  where~$d$ is a non-negative integer, $\alpha\in\mathbb{C}$ is a dominating root of the characteristic polynomial
  \[
    z^h - \eta_1 z^{h-1} - \dots - \eta_{h-1}z - \eta_h,
  \]
  and the bounded real sequence~$(b_n)$ depends on the arguments
  of the dominating characteristic roots; see, e.g.,~\cite{BeGe07} for detailed
  formulas. Now if $|\alpha|\neq1$, or $|\alpha|=1$ and $d\geq1$, then~$|a_n|$ tends to zero or infinity;
  in both cases, the closure of its value set is trivially at most countably infinite. The interesting case is $|\alpha|=1$
  and $d=0$, so that
  \begin{equation}\label{eq:a}
    a_n = b_n + \oh(1).
  \end{equation}
  The sequence~$(b_n)$ can be written as~\cite{BeGe07}
  \begin{equation}\label{eq:b}
    b_n = u_n + v_n,
  \end{equation}
  where~$(u_n)$ is a periodic sequence (corresponding to characteristic roots with arguments
  that are commensurate to~$\pi$), and~$(v_n)$ can be split into~$g$ sections
  \begin{equation}\label{eq:repr}
    v_{gn+k} = \sum_{i=1}^r \xi_i \cos( 2\pi n \sum_{j=1}^m c_{ij} \tau_j + \phi_i), \qquad 0\leq k<g,\ n\geq0,
  \end{equation}
  where~$g$ is a positive integer, $\xi_i$, $\tau_i$, and $\phi_i$ are real numbers,
  $c_{ij}$ are integers, and $1,\tau_1,\dots,\tau_m$ are linearly independent over~$\mathbb{Q}$.
  The numbers~$\xi_i$ and $\phi_i$ depend on~$k$, but~$c_{ij}$ and~$\tau_j$ do not.
  The representation~\eqref{eq:repr} can be established by performing a unimodal
  coordinate change on the arguments of the characteristic roots; see~\cite{BeGe07} for details.

  If the function $F:[0,1]^m \to \mathbb{R}$, also depending on~$k$, is defined by
  \[
    F(t_1,\dots,t_m) := \sum_{i=1}^r \xi_i \cos( 2\pi  \sum_{j=1}^m c_{ij} t_j + \phi_i),
  \]
  then the closure of the set of values of~$(v_{gn+k})_{n\geq0}$ equals the image of~$F$,
  since the sequence of vectors
  \[
    (n\tau_1 \bmod 1,\dots,n\tau_m \bmod 1), \qquad n\geq0,
  \]
  is dense in~$[0,1]^m$ by Kronecker's Approximation Theorem~\cite{Ca57}.
  This image is an interval, since~$F$ is continuous. Now consider a section~$(b_{gn+k})_{n\geq0}$,
  $0\leq k<g$, of the sequence~$(b_n)$. We can split it into sections along which
  the periodic sequence~$(u_n)$ from~\eqref{eq:b} is constant, and the corresponding
  section-of-a-section of~$(v_n)$  will still densely fill the same interval
  as its ``parent''~$(v_{gn+k})_{n\geq0}$.
  The closure of the value set of~$(b_n)$ therefore equals a collection of at most~$g$ closed intervals,
  where the left endpoints are values of the periodic sequence~$(u_n)$ in~\eqref{eq:b}.
  Finally, adding the~$\oh(1)$ part in~\eqref{eq:a} can obviously contribute at most
  countably many values to the closure of the value set of~$(a_n)$, so that the first
  part of the theorem is proved.

  To show the converse statement, let
  \begin{equation}\label{eq:int}
    [\mu_k, \nu_k], \qquad 0\leq k\leq s,
  \end{equation}
  be an arbitrary collection of real intervals. Suppose that $\rho_0:=\nu_0-\mu_0$ is minimal among
  the numbers $\rho_k:=\nu_k-\mu_k$,
  and define $m_k:=\lfloor \rho_k/\rho_0 \rfloor\geq1$. Let~$(w_n)$ be a periodic recurrence sequence
  whose value set comprises exactly the $\sum m_k + s + 1$ numbers
  \begin{align}
    \mu_k + i\rho_0,&\qquad  0\leq i <m_k,\ 0\leq k\leq s, \label{eq:perval1} \\
    \nu_k - \rho_0,&\qquad  0\leq k\leq s. \label{eq:perval2}
  \end{align}
  The existence of such a~$(w_n)$ follows from the fact that periodic interlacements
  of recurrence sequences are themselves recurrence sequences~\cite{EvPoSh03}.
  Define furthermore
  \[
    x_n := \tfrac12 \rho_0 (\cos n + 1), \qquad n\geq 0,
  \]
  which is dense in~$[0,\rho_0]$ by Kronecker's Theorem, and
  \[
    a_n := w_n + x_n, \qquad n\geq0.
  \]
  To see that these are indeed recurrence sequences, recall~\eqref{eq:reccos} and the property that recurrence
  sequences form an  $\mathbb{R}$-algebra~\cite{EvPoSh03}.
  Since any section of~$(x_n)$ is dense in~$[0,\rho_0]$, too, each section of~$(a_n)$
  along which~$(w_n)$ equals one of the values in~\eqref{eq:perval1}--\eqref{eq:perval2}, let us call it $\lambda$,
  densely fills the interval $[\lambda,\lambda+\rho_0]$.
  By the definition of~$m_k$ and the minimality of~$\rho_0$, we have
  \[
    \mu_k \leq \nu_k - \rho_0 < \mu_k + m_k \rho_0 \leq \nu_k, \qquad 0\leq k\leq s.
  \]
  This shows that
  \[
    [\mu_k,\nu_k] = [\nu_k-\rho_0,\nu_k] \cup \bigcup_{i=0}^{m_k-1} [\mu_k + i\rho_0, \mu_k + (i+1)\rho_0],
  \]
  hence the closure of the value set of~$(a_n)$ is given by~\eqref{eq:int}.
\end{proof}

We remark that the following refined version of Kronecker's Approximation Theorem~\cite{Ca57}
could also have been used:
For real numbers $\theta_1,\dots,\theta_N$, without any independence assumption, the sequence
\[
  (n \theta_1 \bmod 1,\dots, n\theta_N \bmod 1), \qquad n\geq0,
\]
is dense in the intersection of the unit cube with finitely many translations of an
affine set. Applying this directly to the arguments of the characteristic roots
provides an alternative to the unimodal coordinate change employed in the proof above.

\bibliographystyle{siam}
\bibliography{../gerhold,../algo}

\end{document}